\documentclass[12pt]{amsart}

\usepackage{fullpage}
\usepackage{amsthm, amsmath, amssymb,amscd}
\usepackage[colorlinks=true, linkcolor=violet, citecolor=blue, urlcolor=brown]{hyperref}
\usepackage{enumitem, booktabs}

\theoremstyle{plain}
\newtheorem{theorem}{Theorem}[section]
\newtheorem{lemma}[theorem]{Lemma}

\newtheorem{corollary}[theorem]{Corollary}
\newtheorem{proposition}[theorem]{Proposition}
\theoremstyle{definition}

\newtheorem{example}[theorem]{Example}

\theoremstyle{remark}
\newtheorem{remark}{\textbf{Remark}}

\numberwithin{equation}{section}
\usepackage{todonotes}

\theoremstyle{definition}

\usepackage{tikz}
\tikzstyle{vertex}=[circle, draw, inner sep=0pt, minimum size=3pt]

\allowdisplaybreaks
\DeclareMathOperator{\ind}{ind}
\newcommand{\z}{\mathbb Z}
\newcommand{\q}{\mathbb Q}

\begin{document}

	\title[ Mnogeneity of reciprocal polynomials]
	{Necessary and sufficient condition for reciprocal polynomials to be monogenic}


	\author{Anuj Narode}
	\address{Department of Mathematics, Indian Institute of Technology Guwahati, Assam, India, PIN- 781039}
	\email{anujnarode@gmail.com \\
		\href{https://orcid.org/0009-0005-6643-3401}{ORCID: 0009-0005-6643-3401}}
	

	\date{\today}
	
	\thanks{}
	
	
	
	\dedicatory{}
	\begin{abstract}
    Let $\mathbb{Z}_K$ denote the ring of integers of the number field $K=\mathbb{Q}(\theta)$, where $\theta$ is a root of a monic irreducible polynomial $f(x)\in\mathbb{Z}[x]$. We say that $f(x)$ is monogenic if $\mathbb{Z}_K=\mathbb{Z}[\theta]$. A polynomial $f(x)\in\mathbb{Z}[x]$ is called reciprocal if $f(x)=x^{\deg(f)}f(1/x)$.
    In this article, we establish necessary and sufficient conditions for the monogeneity of  reciprocal polynomials. As an application, we obtain an alternative  and much simpler proof that the maximal real subfields of cyclotomic fields are monogenic.
\end{abstract}
	\maketitle

	

\section{introduction and Statement of results}
    A number field $K$ is said to be \textit{monogenic} if its ring of integers $\mathbb{Z}_K$ is generated as a $\mathbb{Z}$-algebra by a single element in $\mathbb{Z}_K$, i.e.,$$\mathbb{Z}_K = \mathbb{Z}[\alpha], \text{ for some } \alpha \in \mathbb{Z}_K.$$
    
    Thus, $K = \mathbb{Q}(\alpha)$ admits a \textit{power integral basis} $\{ 1, \alpha, \ldots ,\alpha^{n-1} \}$, where $n = [K : \mathbb{Q}]$.
    Determining whether a number field is monogenic is a longstanding
    and challenging problem in algebraic number theory.
    
    Let $f(x) \in \mathbb{Z}[x]$ be a monic irreducible polynomial of degree $n$ having a root $\theta$, and set $K = \mathbb{Q}(\theta)$.
    The \textit{index }of a polynomial $f(x),$ denoted by  $\operatorname{ind}(f)$, is defined as  the index of $\mathbb{Z}[\theta]$ in $\mathbb{Z}_K$, i.e., $\operatorname{ind}(f)=[ \mathbb{Z}_K : \mathbb{Z}[\theta] ]$.
    Since $\mathbb{Z}_K$ is a free $\mathbb{Z}$-module of rank $\deg(f)$, and $\mathbb{Z}[\theta]$ is a submodule of $\mathbb{Z}_K$ of the same rank, this index is finite. The polynomial $f(x)$ is called monogenic if its index is one, or equivalently  $[ \mathbb{Z}_K : \mathbb{Z}[\theta] ] = 1$.  The monogenity of the polynomial $f(x)$ implies the monogenity of the number field $K$; however, the converse is not true.
    The discriminant of the polynomial  $\Delta(f(x))$, the discriminant of the number field $\Delta(K)$, and the index of the polynomial are linked by the fundamental relation:
     \begin{equation}
	     \Delta(f(x)) = \operatorname{ind}(f)^2 \cdot \Delta(K). \label{p5-Eq-1.1}
	   \end{equation}
    Clearly,  if $\Delta(f(x))$ is squarefree then $f(x)$ is monogenic. In this direction, for any integer $n \ge 2$, Kedlaya in \cite{p2-Kedlaya} constructed an infinite class of monic irreducible polynomials of degree $n$ with integer coefficients having squarefree discriminants. Such polynomials are necessarily monogenic. Further, by extending Kedlaya's approach, for any odd prime $p$, Jones in \cite{p2-Jones_2019} constructed a class of monogenic polynomials of degree $p$ with non-squarefree discriminants. 
       
       The monogenity of a number field significantly simplifies arithmetic computations, such as the calculation of discriminants, integral bases, and ideal arithmetic.
       The problem of determining whether a number field is monogenic is classical and dates back to Dedekind and Hasse. Dedekind introduced the first example of a non-monogenic number field, showing that the cubic field $K =\mathbb{Q}(\alpha)$, where $\alpha$ is a root of $x^3 -x^2 -2x - 8$, is not monogenic (cf. \cite{p5-Dedekind}).
       In 1878, Dedekind introduced a widely used theorem known as \textit{Dedekind's criterion}, which gives a necessary and sufficient condition for a prime $p$ not to divide the index $[\mathbb{Z}_K : \mathbb{Z}[\theta]]$ in terms of the factorization of a polynomial $f(x)$ modulo $p$ (cf. \cite[Theorem 6.1.4]{p5-cohen}).
       This criterion has garnered significant interest, leading to several equivalent versions and generalizations (cf. \cite{p5-Ershov, p5-Khanduja-2016}).
       For further developments in this direction, we refer the reader to the survey article by Ga\'al \cite{p5-gaal-2}.

       In this article, we study monogeneity of reciprocal polynomials. Recall, a polynomial $f(x) \in \mathbb{Z}[x]$ is called reciprocal if $f(x) = x^{\operatorname{deg}(f)} f(1/x)$. For example, cyclotomic polynomials are reciprocal polynomials. Let $n \geq 3$ be an odd integer. Let $f(x)$ be a reciprocal polynomial of degree $n$. Then, $-1$ is a root of $f(x)$. Hence, $f(x) = (x+1) h(x)$, where $h(x)$ is a reciprocal polynomial of degree $n-1$. Therefore, to study reciprocal polynomials it is sufficient to consider even degree reciprocal polynomials. In \cite{p3-Alexandersson}, Alexandersson \textit{et al.} studied the connection between even degree reciprocal polynomials and the Chebyshev polynomials of the first kind. The Chebyshev polynomials of the first kind are defined recursively as follows: 
	$$T_0(x) = 1, ~~ T_1(x) = x ~~ \text{and} ~~ T_n(x) = 2xT_{n-1}(x) - T_{n-2}(x) ~~ \text{for} ~~ n \geq 2.$$ 
We recall the following result of Alexandersson \textit{et al.} \cite{p3-Alexandersson}.
		\begin{proposition} \cite[Proposition 3.2]{p3-Alexandersson}\label{p3-lemma_1}
		Let $n \geq 2$ be an integer and let 
		$$ f(x) = \sum_{j = 0}^{2n} a_j x^j \in \mathbb{Z}[x],$$ 
		where $a_j = a_{2n-j}$ for  all $j$, so that $f(x)$ is reciprocal. Define an nth degree polynomial $g(u) \in \mathbb{Z}[u]$, by 
		\begin{equation} \label{p3-eq-2.2}
			g(u) := a_n + 2 \sum_{j =1}^{n} a_{n-j} T_j(u/2).
		\end{equation}
		Then, $f(x) = x^ng(x+ x^{-1})$.
	\end{proposition}
	We shall henceforth use the following notation. Let $f(x)$ be a reciprocal polynomial of degree $2n$, and let $g(x)$ be the degree $n$ polynomial obtained in Proposition~\ref{p3-lemma_1} such that 
	$$
	f(x) = x^{n} g(x + x^{-1}).
	$$
    In \cite{p3-jones_2021-2}, Jones constructed infinite families of monic sextic reciprocal monogenic polynomials $f(x)$ with $\operatorname{Gal}(f)\cong D_n$, where $\operatorname{Gal}(f)$ denotes the Galois group of $f(x)$ over $\mathbb{Q}$ and $D_n$ denotes the dihedral group of order $2n$, with $n\in\{3,6\}$. Subsequently, in \cite{p3-jones_2021-1}, he extended some of these results to reciprocal polynomials of larger degree. In particular, using cyclotomic polynomials, he constructed a class of monogenic reciprocal polynomials of degree $2^aq^b$, where $a$ and $b$ are positive integers and $q\in\{3,5,7\}$.

Building on this work, Barman, Narode, and Wagh obtained more general sufficient conditions for the monogeneity of reciprocal polynomials. A key ingredient in their proof is the relation
\begin{equation*}
\Delta(f)=(-1)^{n(2n-1)}f(1)f(-1)\Delta(g)^2.
\end{equation*}
They also used this relation to prove a conjecture of Jones concerning the discriminant of reciprocal polynomials (see \cite[Conjecture~3.3]{p3-jones_2021-1}). In particular, they proved the following theorem.

\begin{theorem} \cite[Theorem~1.5]{BNW} \label{p3-th-2.14}
Let $n\geq2$ be an integer, and let
$$
f(x)=\sum_{j=0}^{2n}a_jx^j\in\mathbb{Z}[x]
$$
be a monic irreducible polynomial. Suppose that $a_j=a_{2n-j}$ for all $j$, so that $f(x)$ is reciprocal. Define the polynomial $g(x)\in\mathbb{Z}[x]$ of degree $n$ by
$$
g(x):=a_n+2\sum_{j=1}^{n}a_{n-j}T_j\left(\frac{x}{2}\right).
$$
Then $f(x)$ is monogenic provided that:
\begin{enumerate}
\item[(1)] $f(-1)f(1)$ is squarefree;
\item[(2)] $g(x)$ is monogenic.
\end{enumerate}
\end{theorem}

In the same article, the authors posed the question of whether the converse of Theorem~\ref{p3-th-2.14} holds. We show that, in general, the converse is false. However, if condition~(1) is replaced by the stronger requirement that $f(1)$ and $f(-1)$ are individually squarefree, then the converse does hold. The following example serves as a counter example to Theorem~\ref{p3-th-2.14}.
    
        \begin{example}
            Consider $$f(x)=x^4+3x^2+1,\qquad g(t)=t^2+1.$$
        Then $f(x)=x^2g(x+x^{-1})$ and $f(1)=f(-1)=5$, but $f(x)$ is monogenic.
        
        We first show that $f(x)$ is irreducible. Note that $f(x)$ does not have a rational root. A factorization into monic integral quadratics would have the form
        $$(x^2+ax+b)(x^2-ax+b),\qquad b\in\{1,-1\},$$
        forcing $3=2b-a^2$, which is impossible. Thus $f(x)$ is irreducible.
        Since $\Delta(f)=400$, only $2$ and $5$ can divide its index. Dedekind's criterion (or see Theorem~\ref{p5-Thm-2.1}) gives the following calculations, using the indicated monic lifts.
        \begin{table}[h!]
        \caption{Factorization of $f(x)$ modulo $p$ and the corresponding values of $M(x)$}
        \label{...}
        \centering
        \begin{minipage}{0.7\textwidth}
        \centering
        \renewcommand{\arraystretch}{1.35}
        \begin{tabular}{ccc}
        \hline
        $p$ & Factorization modulo $p$ & $M(x)$\\
        \hline
        $2$ & $(x^2+x+1)^2$ & $-x^3-x$\\
        $5$ & $(x-1)^2(x+1)^2$ & $x^2$\\
        \hline
        \end{tabular}
        \end{minipage}
        \end{table}

        Here $M(x)$ is $1/p$ times the difference between $f(x)$ and the product of the lifted factors with their indicated exponents. In both cases, $\overline M(x)$ is coprime to every repeated irreducible factor. Consequently,
        $$\ind(f)=1,\qquad f(1)f(-1)=25.$$
        This disproves the converse.
        \end{example}
         The obstruction to monogeneity is natural, as $p^2\mid f(1)f(-1)$ may occur because $p\mid f(1)$ and $p\mid f(-1)$, although $p^2\nmid f(1)$ and $p^2\nmid f(-1)$.
         This suggests considering a converse under the modified hypotheses that $g(x)$ is monogenic and $f(1)$ and $f(-1)$ are individually squarefree. Now we state the corrected converse.
\begin{theorem} \label{p6-thm-1.1}
		Let $n \geq 2$ be an integer, and let 
		$$f(x) = \sum_{j=0}^{2n} a_j x^j \in \mathbb{Z}[x]$$ 
		be a monic and irreducible polynomial. Let $a_j = a_{2n-j}$ for all $j$, so that $f(x)$ is reciprocal. Define an $n$th degree polynomial $g(x) \in \mathbb{Z}[x]$ by 
		$$g(x) := a_n + 2\sum_{j=1}^{n} a_{n-j} T_j\left(\frac{x}{2}\right).$$ 
		Then, $f(x)$ is monogenic if and only if $g(x)$ is monogenic and $f(-1)$, $f(1)$ are individually squarefree.
	\end{theorem}
Let $\Phi_m(x)$ denote the $m$-th cyclotomic polynomial, where $m\geq2$. It is well known that $\Phi_m(x)$ is reciprocal. Let $\zeta_m$ be a primitive $m$-th root of unity. Since the ring of integers of the cyclotomic field $\mathbb{Q}(\zeta_m)$ is $\mathbb{Z}[\zeta_m]$, the cyclotomic polynomial $\Phi_m(x)$ is monogenic. Let $g(x)$ be a polynomial of degree $\phi(m)/2$ such that
    $$\Phi_m(x)=x^{\phi(m)/2}g(x+x^{-1}).$$
    By Theorem~\ref{p6-thm-1.1}, $g(x)$ is monogenic. Moreover, if $\eta=\zeta_m+\zeta_m^{-1}$, then $g(\eta)=0$, and hence the number field generated by a root of $g(x)$ is
    $$K_m^+=\mathbb{Q}(\zeta_m+\zeta_m^{-1}).$$
    Therefore, $K_m^+$ is monogenic and
    $$\mathbb{Z}_{K_m^+}=\mathbb{Z}[\zeta_m+\zeta_m^{-1}].$$
    This provides an alternative proof of a theorem of Liang \cite[Theorem~4]{Liang}.
\begin{corollary}
    Let $m\geq3$ and $K_m=\mathbb{Q}(\zeta_m)$. Then the maximal real subfield $K_m^+=\mathbb{Q}(\zeta_m+\zeta_m^{-1})$ is monogenic.
\end{corollary}
In \cite{p3-Jones-2022, p3-Jones-2025}, Jones studied the monogeneity of a class of reciprocal polynomials obtained by composing $x^{2^{n-1}}$ with certain reciprocal polynomial. More recently Kaur, Kumar, and Remete~\cite{p3-Kaur-2025}  studied the index of power  compositional polynomials. By combining their results with Theorem~\ref{p6-thm-1.1}, we obtain the following corollary for compositional reciprocal polynomials.
	\begin{corollary} \label{p3-Prop-1.7}
		Let $f(x)$ be a monic irreducible reciprocal polynomial of degree $2n$ with integer coefficients. Let $k \geq 2 $ be an integer such that $f(x^k) $ is irreducible. Then $f(x^k)$ is monogenic  if and only if the following conditions holds:
		\begin{enumerate}
			\item $g(x)$ is monogenic,
			\item $f(-1)$ and $f(1) $ are individually  squarefree, 
			\item $p$ does not divide the index of $f(x^k)$ for all primes $p \mid k$.
		\end{enumerate}
	\end{corollary}
	\begin{remark}
		 Theorem~\ref{p6-thm-1.1} and Corollary~\ref{p3-Prop-1.7} highlight a significant reduction in complexity: the problem of establishing the monogeneity of the degree $2n$ polynomial $f(x)$ reduces to the simpler task of verifying the monogeneity of the degree $n$ polynomial $g(x)$.
	\end{remark}
\section{Preliminaries}
    In this section, we recall some basic properties of reciprocal polynomials, an equivalent formulation of Dedekind's criterion, and Uchida's theorem for Dedekind rings. Recall that $f(x)$ and $g(x)$ are related by the identity
    \begin{equation} \label{Eq-2.1}
        f(x) = x^{n} g(x+x^{-1}).
    \end{equation} With this notion one can observe that $f(1) = g(2) $ and $f(-1) = (-1)^ng(-2)$. Also, differentiating  \eqref{Eq-2.1} gives,
    \begin{align*}
        f'(x) & = nx^{n-1}g(x+x^{-1}) +  x^n g'(x+x^{-1})(1 - 1/x^2).  
    \end{align*}
    Thus $$ f'(1) = ng(2) = nf(1)  \text{ and } f'(-1)  = n(-1)^{n-1}g(-2)= -nf(-1).$$
     Now, we recall an equivalent formulation of Dedekind's criterion \cite[Theorem~6.1.4]{p5-cohen}. Various generalizations of this criterion have been studied in the literature.
     In particular, Khanduja and Jhorar \cite{p5-Khanduja-2016} established equivalent formulations of the generalized Dedekind criterion. We use the following formulation in the proof of Lemma~\ref{lemma-1.5}; see also \cite[Lemma~2.1]{p5-MR3486262}.
	 \begin{theorem} \cite[Theorem 1.1]{p5-Khanduja-2016} \label{p5-Thm-2.1}
			Let $f(x) \in \mathbb{Z}[x]$ be a monic irreducible polynomial having the factorization $\overline{g}_1(x)^{e_1} \cdots \overline{g}_t(x)^{e_t}$ modulo a prime $p$  as a product of powers of distinct irreducible polynomials over $\mathbb{Z}/p\mathbb{Z}$ with each $g_i(x) \in \mathbb{Z}[x] $ monic. Let $K = \mathbb{Q}(\theta)$ with $\theta$ a root of $f(x)$. The the following are equivalent:
			\begin{enumerate}[label=\textup{(\roman*)}]
				\item $p$ does not divide $[\mathbb{Z}_K : \mathbb{Z}[\theta]]$;
				\item for each $i$, either $e_i = 1$ or $\overline{g}_i(x)$ does not divide $\overline{M}(x)$ where, 
							\begin{equation*}
								M(x) = \frac{1}{p} \left(f(x) - \ g_1(x)^{e_1} \cdots g_t(x)^{e_t}\right);
							\end{equation*}
				\item $f(x) $ does not belong to the ideal $\langle p, g_i(x) \rangle^2$ in $\mathbb{Z}[x]$ for any $i$, $1 \leq i \leq t.$
			\end{enumerate}
		\end{theorem}
        The following theorem, due to Uchida, plays a key role in establishing the monogeneity results of this article.
        \begin{theorem} \cite[Theorem]{uchida}\label{uchida}
            Let $R$ be a Dedekind ring. Let $\theta$ be an element of some integral domain which contains $R$, and let $\theta$ be integral over $R$. Then $R[\alpha]$ is a Dedekind ring if and only if defining polynomial $f(x)$ of $\alpha$ is not contained in $\mathfrak{m}^2$ for any maximal ideal $\mathfrak{m}$ of the polynomial ring $R[x]$.
        \end{theorem}
        Uchida also described the structure of maximal ideal in the Dedekind ring.
    \begin{lemma} \cite[Lemma]{uchida} \label{lemma-uchida}
        Let $\mathfrak{m}$ be a maximal ideal of $R[x]$. If $\mathfrak{m}$ contains an integral polynomial, then $\mathfrak{m}$ is of the form $\mathfrak{m} = (\mathfrak{p},f(x))$, where $\mathfrak{p}$ is a maximal ideal of $R$ and $f(x)$ is an integral polynomial which is irreducible modulo $\mathfrak{p}$.
    \end{lemma}
    Theorem~\ref{uchida} and Lemma~\ref{lemma-uchida} are useful in proving Lemma~\ref{lemma-1.5}.
	 \section{Proof of Theorem~\ref{p6-thm-1.1}}
In this section, we prove Theorem~\ref{p6-thm-1.1} using three crucial lemmas. The first establishes that $p^2\mid f(1)$ or $p^2\mid f(-1)$ implies $p\mid\operatorname{ind}(f)$. The second gives an algebraic relation between $\operatorname{ind}(f)$ and $\operatorname{ind}(g)$, thereby strengthening Proposition~1.8 of \cite{BNW}. The final lemma, assuming that $g(x)$ is monogenic, relates the prime divisors of $\operatorname{ind}(f)$ to those of $f(1)$ and $f(-1)$. The main tools used in the proofs of these lemmas are equivalent formulations of Dedekind's criterion and a theorem of Uchida.
    
    \begin{lemma} \label{lemma-1.5}
        Let $f(x) \in \mathbb{Z}[x]$ be a monic irreducible  reciprocal polynomial of degree $2n$ and let $\varepsilon \in \{1,-1\}$. For a prime $p$, if $p^2 \mid f(\varepsilon)$ then $p \mid \operatorname{ind}(f)$.
        \end{lemma}
    \begin{proof}
        As $f(x)$ is a reciprocal polynomial it satisfies $f(x) = x^ng(x + x^{-1})$. Differentiating this identity gives, 
        \begin{equation} \label{p6-eq-1.1}
            f'(1)=nf(1),\qquad f'(-1)=-nf(-1).
        \end{equation}
        The Taylor expansion in $\mathbb{Z}[x]$ of $f(x)$ around $\varepsilon$ is $$f(x)=f(\varepsilon)+f'(\varepsilon)(x-\varepsilon)
             +(x-\varepsilon)^2P_{\varepsilon}(x),
            \qquad P_{\varepsilon}\in\mathbb{Z}[x].$$
            Using \eqref{p6-eq-1.1} yields, 
            $$f(x)=f(\varepsilon)+n \varepsilon f(\varepsilon)(x-\varepsilon)
             +(x-\varepsilon)^2P_{\varepsilon}(x).$$
            Thus by Theorem~\ref{p5-Thm-2.1},
            \begin{equation}\label{Eq-1.1}
            p^2\mid f(\varepsilon)
            \ \Longrightarrow\ f\in(p,x-\varepsilon)^2
            \ \Longrightarrow\ p\mid\operatorname{ind}(f).
            \end{equation}
    \end{proof}
    \begin{lemma} \label{lemma-1.6}
        Let $f(x)$ and $g(x)$ are as defined in Theorem \ref{p6-thm-1.1}. Let $\theta$ be a root of $f(x)$, and put
        $$K=\mathbb{Q}(\theta),\qquad \alpha=\theta+\theta^{-1},\qquad L=\mathbb{Q}(\alpha).$$ Then $$\operatorname{ind}(f) = \operatorname{ind}(g)^2 [\mathbb{Z}_K : \mathbb{Z}_L[\theta]].$$
    \end{lemma}
    \begin{proof}
        
        Since $g(\alpha)=0$ and $\theta$ satisfies $X^2-\alpha X+1$, the degree bounds and $[K:\mathbb{Q}]=2n$ imply that $[L:\mathbb{Q}]=n$ and $[K:L]=2$. In particular, $g(x)$ is the minimal polynomial of $\alpha$ and $\theta \not\in L$.
        Note that  $f(x)$ has constant term $1$ and $f(\theta)  = \theta^{2n} +a_{2n-1} \theta^{{2n-1}}+ a_1 \theta +\cdots+ 1 = 0$.
        Dividing by $\theta^{-1}$ gives,
        $$\theta^{-1} = -(\theta^{2n-1} +a_{2n-1} \theta^{{2n-2}}+ \cdots+  a_1) .$$
        Consequently, $$\alpha \in \mathbb{Z}[\theta], \quad\quad  \mathbb{Z}[\alpha][\theta] = \mathbb{Z}[\theta].$$
        As, $\theta^2 -\alpha \theta +1 = 0$ we get $\theta^2 =  \alpha \theta - 1$.
        Since $\alpha \in \mathbb{Z}[\alpha]$ every higher power of $\theta$ can be written as reduced  $\mathbb{Z}[\alpha]$ linear combination of 1 and $\theta$. 
        Explicitly if $\theta^m = a + b\theta$ where $a,b \in \mathbb{Z}[\alpha]$ then $\theta^{m+1} = a\theta + b\theta^2 = -b + (a+\alpha b)\theta$.
        Induction therefore shows that $$\mathbb{Z}[\alpha][\theta] = \mathbb{Z}[\alpha] + \theta\mathbb{Z}[\alpha].$$
        To show that this sum is direct, suppose $ a + b\theta = 0$ where $a,b \in \mathbb{Z}[\alpha]$. If $b \not= 0$ then $\theta = -a/b \in  L$, as $a,b \in \mathbb{Z}[\alpha]$, this implies that $\theta \in \mathbb{Z}[\alpha]$. This gives a contradiction as $\theta \not \in L$. Hence $b = 0 $, and then $a = 0$. Thus the sum is direct. Similarly we can show that $\mathbb{Z}_L[\theta] = \mathbb{Z}_L \oplus \theta\mathbb{Z}_L.$ Thus we have proved \begin{equation} \label{eq-1.2}
            \mathbb{Z}[\alpha][\theta] = \mathbb{Z}[\alpha] \oplus \theta\mathbb{Z}[\alpha] \quad \text{and}\quad \mathbb{Z}_L[\theta] = \mathbb{Z}_L \oplus \theta\mathbb{Z}_L.
        \end{equation}

        Next consider the intermediate index $[\mathbb{Z}_L[\theta] : \mathbb{Z}[\theta]]$. The decompositions in \eqref{eq-1.2} yield an isomorphism of additive groups
        \begin{equation} \label{eq-1.3}
            \frac{\mathbb{Z}_L[\theta]}{\mathbb{Z}[\theta]} \cong \frac{\mathbb{Z}_L \oplus \theta \mathbb{Z}_L}{\mathbb{Z}[\alpha] \oplus \theta \mathbb{Z}[\alpha]} \cong \frac{\mathbb{Z}_L}{\mathbb{Z}[\alpha]} \oplus \frac{\mathbb{Z}_L}{\mathbb{Z}[\alpha]}.
        \end{equation}
        Since $g(x)$ is the minimal polynomial of $\alpha$, we have $\operatorname{ind}(g) = [\mathbb{Z}_L:\mathbb{Z}[\alpha]].$
        Taking cardinalities in \eqref{eq-1.3}, we obtain
        \begin{equation} \label{eq-1.4}
            \operatorname{ind}(g)^2 = [\mathbb{Z}_L[\theta]:\mathbb{Z}[\theta]].
        \end{equation}
        Finally, the inclusions
        \[
        \mathbb{Z}[\theta]\subseteq \mathbb{Z}_L[\theta]\subseteq \mathbb{Z}_K
        \]
        are inclusions of additive groups of finite index. By the multiplicativity of
        indices and \eqref{eq-1.4}, we obtain
        $$
        \begin{aligned}
        \operatorname{ind}(f)
        &=[\mathbb{Z}_K:\mathbb{Z}[\theta]]\\
        &=[\mathbb{Z}_K:\mathbb{Z}_L[\theta]]
          [\mathbb{Z}_L[\theta]:\mathbb{Z}[\theta]]\\
        &=\operatorname{ind}(g)^2
          [\mathbb{Z}_K:\mathbb{Z}_L[\theta]].
        \end{aligned}
        $$
        Therefore,
       $$
        \operatorname{ind}(f)
        =
        \operatorname{ind}(g)^2
        [\mathbb{Z}_K:\mathbb{Z}_L[\theta]].
        $$
    \end{proof}
    \begin{remark}
        Lemma~\ref{lemma-1.6} implies that $\operatorname{ind}(g)^2 \mid \operatorname{ind}(f) .$ This strengthens  Proposition~1.8 of \cite{BNW}.
    \end{remark}
    Next corollary is an immediate cosequence of Lemma~\ref{lemma-1.6}. 
    \begin{corollary}
        Let $f(x)$ and $g(x)$ are as defined in Theorem \ref{p6-thm-1.1}. If $f(x)$ is monogenic then $g(x)$ is monogenic.
    \end{corollary}
    \begin{lemma} \cite[Equation 3.3]{BNW} \label{p6-lemma1.3}
        Let $f(x)$ and $g(x)$ are as defined in Theorem \ref{p6-thm-1.1}.  Then $$\ind(f)^2 \mid f(-1)f(1) \ind(g)^4.$$ 
    \end{lemma}
    \begin{lemma} \label{lemma-2.5}
        Let $f(x)$ and $g(x)$ are as defined in Theorem \ref{p6-thm-1.1}. Suppose $g(x)$ is monogenic, then $$p \mid \operatorname{ind}(f) \iff  p^2 \mid f(1) \text{ or } p^2 \mid f(-1).$$
    \end{lemma}
    \begin{proof}
        As $g(x)$ is monogenic we have $\mathbb{Z}_L = \mathbb{Z}[\alpha]$. We apply Uchdia's criterion (see Theorem~\ref{uchida}) over the Dedekind domain $\z_L$ to $h(x) = x^2 -\alpha x +1$. Suppose  $p \mid \ind(f)$. Let $\mathfrak{p}$ be  a prime ideal of $\z_L$ above the rational prime $p$. 
        A repeated factor of $h(x)$ modulo $\mathfrak{p}$ occurs only when
        $$\alpha\equiv2\varepsilon\pmod{\mathfrak{p}},\qquad \varepsilon\in\{1,-1\}.$$
        Then $h(x) \equiv (x-\varepsilon)^2 \pmod{\mathfrak{p}}.$ Expanding $h(x)$ in terms of $x -\varepsilon$ we get,
        \begin{equation}
            h(x)=(x-\varepsilon)^2+(2\varepsilon-\alpha)(x-\varepsilon)+\varepsilon(2\varepsilon-\alpha).
        \end{equation}
        Since $2\varepsilon-\alpha\in\mathfrak{p}$, comparison of the constant and linear coefficients in powers of $x-\varepsilon$ gives
        \begin{equation}\label{eq-1.6}
            h(x)\in(\mathfrak{p},x-\varepsilon)^2
            \quad\Longleftrightarrow\quad
            \alpha-2\varepsilon\in\mathfrak{p}^2.
        \end{equation}
        
        As $p \mid \ind(f)$ and $g(x)$ is monogenic by Lemma~\ref{p6-lemma1.3} we get $p \mid g(2 \varepsilon)$.
        Now we show that there is exactly one prime ideal of residue degree one, above $p$ containing $\alpha- 2\varepsilon$, namely $\mathfrak{p}_{\varepsilon}= (p, \alpha- 2\varepsilon)$.
        Set $A = \z_L =\z[\alpha]$, $a = 2\varepsilon$ and $J =  (p, \alpha- a)$. Since $g(x)$ is the minimal polynomial of $\alpha$ we get, $$ A = \z[x]/(g(x) ) .$$
        Then \begin{align*}
        A/J & = \z[x]/(g(x), p , x -a ) \cong \mathbb{F}_p/(\overline{g}(a)),
        \end{align*}
        as $p \mid g(a)$, $\overline{g}(a) = 0$ in $\mathbb{F}_p$, therefore $A/J \cong \mathbb{F}_p$. This also shows that the residue degree of $J$ is 1.
        
        As $g(x)$ is the  minimal polynomial of $\alpha$ we get $N_{L/\q}(\alpha) =  g(0)$. This implies that \begin{equation} \label{Eq-1.7}
            N_{L/\q}(\alpha -2\varepsilon) =  (-1)^n g(2\varepsilon).
        \end{equation} 
        By \eqref{Eq-1.7} we get,
        \begin{equation}\label{eq-1.8}
            v_{\mathfrak{p}_\varepsilon}(\alpha-2\varepsilon)
            = (\text{residue degree of }\mathfrak{p}) \cdot  v_p(g(2\varepsilon))= 1 \cdot v_p(g(2\varepsilon))=v_p(f(\varepsilon)).
        \end{equation}
        Thus combining \eqref{eq-1.6} and \eqref{Eq-1.7} yields,
        \begin{equation} \label{eq-1.9}
            \alpha-2\varepsilon\in\mathfrak{p}^2 \iff v_{\mathfrak{p}_\varepsilon}(\alpha-2\varepsilon) \geq 2 \iff v_p(f(\varepsilon)) \geq 2.
        \end{equation}
        When a reduction of $h(x)$ has no repeated factor, it creates no obstruction in Uchida's criterion. Therefore \eqref{eq-1.6} and \eqref{eq-1.9} yields, for monogenic $g(x)$
        \begin{equation}
            p \mid \ind(f) \implies p^2 \mid f(-1) \quad \text{or} \quad p^2 \mid f(1).
        \end{equation}

        Converse easily follows  from Lemma~\ref{lemma-1.5}. For proving converse part we do not require $g(x)$ to be monogenic.
\end{proof}
\begin{proof}[\textbf{Proof of Theorem~\ref{p6-thm-1.1}}]
    Suppose that $g(x)$ is monogenic and that $f(1)$ and $f(-1)$ are individually squarefree. By Lemma~\ref{lemma-2.5}, $\operatorname{ind}(f)=1$, and hence $f(x)$ is monogenic. Conversely, if $f(x)$ is monogenic, then Lemma~\ref{lemma-1.6} implies that $g(x)$ is monogenic. Therefore, both $f(x)$ and $g(x)$ are monogenic, and Lemma~\ref{lemma-2.5} yields that $f(1)$ and $f(-1)$ are individually squarefree.
\end{proof}
    
    To prove Corollary~\ref{p3-Prop-1.7}, we recall a theorem of Kaur, Kumar, and Remete concerning the monogeneity of power compositional polynomials.
\begin{theorem} \cite[Theorem 1.1]{p3-Kaur-2025} \label{p3-prop-4.11}
	Let $f(x)$ be a monic polynomial with integer coefficients. Let $k \geq 2 $ be an integer such that $f(x^k) $ is irreducible. Then $f(x^k)$ is monogenic if and only if the following conditions hold.
	\begin{enumerate}
		\item $f(x)$ is monogenic,
		\item $f(0)$ is squarefree,
		\item $p$ does not divide the index of $f(x^k)$ for all primes $p \mid k$.
	\end{enumerate}
\end{theorem}
	\begin{proof}[\textbf{Proof of Corollary \ref{p3-Prop-1.7}}]
		Since $f(x)$ is monic and reciprocal, $f(0)=1$. Furthermore, condition~(1) of Theorem~\ref{p3-prop-4.11} can be replaced by the assumptions that $g(x)$ is monogenic and $f(1)$ and $f(-1)$ are individually squarefree. Hence, the result follows from Theorems~\ref{p6-thm-1.1} and~\ref{p3-prop-4.11}.
	\end{proof}
	 

\begin{thebibliography}{10}

\bibitem{p3-Alexandersson}
Per Alexandersson, Luis~Angel Gonz\'alez-Serrano, Egor~A. Maximenko, and Mario~Alberto Moctezuma-Salazar, \emph{Symmetric polynomials in the symplectic alphabet and the change of variables {$z_j = x_j + x^{-1}_j$}}, Electron. J. Combin. \textbf{28} (2021), no.~1, Paper No. 1.56, 36. \MR{4245289}

\bibitem{BNW}
Rupam Barman, Anuj Narode, and Vinay Wagh, \emph{On monogeneity of reciprocal polynomials}, Ramanujan J. \textbf{69} (2026), no.~3, Paper No. 56, 17. \MR{5031649}

\bibitem{p5-cohen}
Henri Cohen, \emph{A course in computational algebraic number theory}, Graduate Texts in Mathematics, vol. 138, Springer-Verlag, Berlin, 1993. \MR{1228206}

\bibitem{p5-Dedekind}
Richard Dedekind, \emph{Über den zusammenhang zwischen der theorie der ideale und der theorie der höheren kongruenzen}, Abhandlungen der Königlichen Gesellschaft der Wissenschaften zu Göttingen \textbf{23} (1878), 1--23.

\bibitem{p5-Ershov}
Yu.\~L. Ershov, \emph{The {D}edekind criterion for arbitrary valuation rings}, Dokl. Akad. Nauk \textbf{410} (2006), no.~2, 158--160. \MR{2455373}

\bibitem{p5-gaal-2}
Istv\'an Ga\'al, \emph{Monogenity and power integral bases: recent developments}, Axioms \textbf{13} (2024), no.~7, 429--429.

\bibitem{p5-MR3486262}
Anuj Jakhar, Sudesh~K. Khanduja, and Neeraj Sangwan, \emph{On prime divisors of the index of an algebraic integer}, J. Number Theory \textbf{166} (2016), 47--61. \MR{3486262}

\bibitem{p2-Jones_2019}
Lenny Jones, \emph{Monogenic polynomials with non-squarefree discriminant}, Proc. Amer. Math. Soc. \textbf{148} (2020), no.~4, 1527--1533. \MR{4069191}

\bibitem{p3-jones_2021-1}
\bysame, \emph{Infinite families of reciprocal monogenic polynomials and their {G}alois groups}, New York J. Math. \textbf{27} (2021), 1465--1493. \MR{4334375}

\bibitem{p3-jones_2021-2}
\bysame, \emph{Sextic reciprocal monogenic dihedral polynomials}, Ramanujan J. \textbf{56} (2021), no.~3, 1099--1110. \MR{4341112}

\bibitem{p3-Jones-2022}
\bysame, \emph{Reciprocal monogenic quintinomials of degree {$2^n$}}, Bull. Aust. Math. Soc. \textbf{106} (2022), no.~3, 437--447. \MR{4510135}

\bibitem{p3-Jones-2025}
\bysame, \emph{Reciprocal monogenic septinomials of degree {$2^n3$}}, Ann. Math. Sil. \textbf{39} (2025), no.~1, 155--169. \MR{4895441}

\bibitem{p3-Kaur-2025}
Sumandeep Kaur, Surender Kumar, and L\'aszl\'o Remete, \emph{On the index of power compositional polynomials}, Finite Fields Appl. \textbf{107} (2025), Paper No. 102642, 20. \MR{4902688}

\bibitem{p2-Kedlaya}
Kiran~S. Kedlaya, \emph{A construction of polynomials with squarefree discriminants}, Proc. Amer. Math. Soc. \textbf{140} (2012), no.~9, 3025--3033. \MR{2917075}

\bibitem{p5-Khanduja-2016}
Sudesh~K. Khanduja and Bablesh Jhorar, \emph{When is {$R[\theta]$} integrally closed?}, J. Algebra Appl. \textbf{15} (2016), no.~5, 1650091, 7. \MR{3479450}

\bibitem{Liang}
Joseph~J. Liang, \emph{On the integral basis of the maximal real subfield of a cyclotomic field}, J. Reine Angew. Math. \textbf{286/287} (1976), 223--226. \MR{419402}

\bibitem{uchida}
K\^oji Uchida, \emph{When is {$Z[\alpha ]$} the ring of the integers?}, Osaka Math. J. \textbf{14} (1977), no.~1, 155--157. \MR{450255}

\end{thebibliography}

\end{document}